\documentclass[12pt]{article}

\usepackage{amstext    }
\usepackage{amsthm    }
\usepackage{a4}
\usepackage[mathscr]{eucal}
\usepackage{indentfirst}
\usepackage{amsmath,bm}

\usepackage{amsmath}
\usepackage{amssymb}
\usepackage{amscd}

\newtheorem{theorem}{Theorem}[section]
\newtheorem{definition}[theorem]{Definition}
\newtheorem{proposition}[theorem]{Proposition}

\newtheorem{remark}[theorem]{Remark}

\newcommand{\supp}{{\rm supp}}

\title{Convergence speed of weighted Bergman kernels towards extremal functions}

\author{Guokuan SHAO}

\begin{document}

\maketitle

\begin{abstract}
We construct inner products by the Bernstein-Markov inequality on spaces of holomorphic sections of high powers of a line bundle. 
The corresponding weighted Bergman kernel functions converge to an extremal function.
We obtain a uniform convergence speed.
\end{abstract}

\noindent
{\bf Classification AMS 2010:} 32A36, 32L10, 32Q15, 32U15.


\section{Introduction}
In this paper we study the convergence speed when the weighted Bergman kernel functions for high powers of a line bundle converge to an extremal function. Let $(L, h)\rightarrow (X, \omega)$ be a positive line bundle over a projective manifold $X$ of dimension $m$,
where $h$ is a smooth Hermitian metric of  $L$, $\omega$ is the K\"ahler form of $X$. Denote by $H^{0}(X, L)$ the space of all global holomorphic sections of $L$.
A natural inner product on $H^{0}(X, L)$ is induced by $h$ and $\omega$.
Let $B_{n}$ be the Bergman kernel for $H^{0}(X, L^{n})$. Let $\Psi_{n}: X\rightarrow\mathbb{P}H^{0}(X, L)$ be the Kodaira embedding map when $n$ is large. Denote by $\omega_{FS}$ the Fubini-Study form of complex projective spaces.
A famous theorem by Tian-Zelditch \cite{zs} states that $\dfrac{1}{n}\Psi_{n}^{\star}(\omega_{FS})$ converges uniformly to $\omega$ in $C^{\infty}$
topology. Alternatively we have $\frac{1}{n}\log B_{n}$ converges uniformly to $0$ on $X$.
Since then the convergence of Bergman kernels have been studied actively, see \cite{dlm, hc, hm, mm1} and the references therein.
Many works extended the results in more general settings, e.g. the line bundle is big with singular metrics,
the base space is normal K\"ahler complex space etc, see \cite{cmm, dmm}.
One of the applications of these results is to explore the zeros of random holomorphic sections of line bundles. For related works, see \cite{cm, cmn, ds1, sh1, sh2, sz} etc.

Bloom-Shiffman \cite{bs} considered a Bergman kernel induced by a new inner product. They used Bernstein-Markov inequalities to construct the new inner products. The limit of
$\dfrac{1}{2n}\log B_{n}$ is the Siciak's extremal function \cite{sj}.
The setting of the work is based on the space of homogeneous polynomials on complex vector spaces.
Now we generalize the setting to the case of positive line bundles with a weighted function and obtain a uniform convergence speed.

Some basic settings are the following：\\
\emph{I: Let $(L, h)\rightarrow (X, \omega)$ be a positive line bundle over a projective manifold $X$ of dimension $m$.
The smooth Hermitian metric $h$ of  $L$ induces the K\"ahler form $\omega$ of $X$.}\\
\emph{II:The Bergman kernel function $B_{n}(x)$ for $H^{0}(X, L^{n})$ is defined by (3).
Note that in the Bernstein-Markov inequality, we have the condition $M_{n}\leq n^{C_{0}}$ (cf. Definition 3.1).}\\
\emph{III: The extremal function $V(x)$ is defined by (1).}\\

Here is our main theorem.
\begin{theorem}
With the above assumptions, we have uniformly
\begin{equation*}
|\frac{1}{2n}\log B_{n}(x)-V(x)|=O(\frac{\log n}{n}).
\end{equation*}
\end{theorem}
\begin{remark}
Note that the Bergman kernel in the theorem is different from the classical one.
It depends on the choice of inner products. Different inner products induce different completion spaces of $H^{0}(X, L^{n})$.
\end{remark}

The paper is organized as follows. In Section 2 we recall the notion of holomorphic line bundles. In Section 3 we introduce the extremal functions. We consider the Bernstein-Markov inequality of special type. Then we define new inner products which induce the weighted Bergman kernels. The H\"ormander's $L^{2}$-estimate for $\bar\partial$ is also mentioned. We conclude Section 4 with the proof of the main theorem by two steps.

\section{Holomorphic line bundles}
In this section, we introduce basic notions of holomorphic line bundles.
Let $L\rightarrow X$ be a holomorphic line bundle over a complex compact K\"ahler manifold $X$. The complex dimension of $X$ is $m$. Let $\pi:L\rightarrow X$ be the projection map.
There exist local trivializations of $L$ on an open cover $\{U_{\alpha}\}$ of $X$.
The biholomorphisms are $\Psi_{\alpha}: \pi^{-1}(U_{\alpha})\rightarrow U_{\alpha}\times\mathbb{C}$,
which send $\pi^{-1}(x)$ isomorphically onto $\{x\}\times \mathbb{C}$. The transition functions are defined by the formula
\begin{equation*}
g_{\alpha\beta}=\Psi_{\alpha}\circ\Psi_{\beta}^{-1} \quad\text{on}\quad
U_{\alpha\beta}:=U_{\alpha}\cap U_{\beta}.
\end{equation*}
The $g_{\alpha\beta}$ are non-where vanishing holomorphic functions on $U_{\alpha\beta}$. The C\v{e}ch cohomology class of $\{g_{\alpha\beta}\}$ defines the first Chern class $c_{1}(L)$ of $L$. 

Denote by $H^{0}(X, L)$ the space of all global holomorphic sections of $L$. Let $h$ be a Hermitian metric of $L$, $e_{\alpha}$ the local frame of $L$ on $U_{\alpha}$. For a section $s\in H^{0}(X, L)$, we write $s=s_{\alpha}e_{\alpha}$
on $U_{\alpha}$. Then we can say equally that $s$ is a collection of holomorphic functions $s_{\alpha}$ on $U_{\alpha}$ which subject to the compatibility condition $s_{\alpha}=g_{\alpha\beta}s_{\beta}$ on $U_{\alpha\beta}$.
Set $h(e_{\alpha}, e_{\alpha})=e^{-2\phi_{\alpha}}, \|e_{\alpha}\|_{h}=e^{-\phi_{\alpha}}$, where $\phi_{\alpha}\in L^{1}(U_{\alpha})$.
Then we can say equally that $h$ is a collection of functions $\phi_{\alpha}\in L^{1}(U_{\alpha})$ with the conditions
\begin{equation*}
\phi_{\alpha}=\phi_{\beta}+\log |g_{\alpha\beta}| \quad\text{on}\quad U_{\alpha\beta}.
\end{equation*}
\begin{definition}
A line bundle $L$ is said to be Lipschitz if all $ \phi_{\alpha} $ are Lipschitz functions.
\end{definition}
Positive line bundles are always Lipschitz since they admit smooth Hermitian metrics.
The curvature current
\begin{equation*}
c_{1}(h)=-dd^{c}\log\|e_{\alpha}\|_{h}=dd^{c}\phi_{\alpha}
\end{equation*}
represents the first Chern class $ c_{1}(L)$.
Here $ d^{c}=\frac{i}{2\pi} (\bar\partial-\partial)$.
The line bundle is called $positive$ if $c_{1}(h)$ is strictly positive. In this case we can take $\omega=c_{1}(h)$
to be the K\"ahler form of $X$. The Kodaira embedding theorem implies that $X$ is a projective manifold.
Let $L^{n}$ be the $n$th tensor product of $L$ with the natural metric $h_{n}:=h^{\otimes n}$ induced by $h$.

\section{Extremal functions and weighted Bergman kernels}
In this section, we will introduce the extremal functions and define weighted Bergman kernels.
From now on, we assume $(L, h)\rightarrow (X, \omega)$ is a positive line bundle over a projective manifold $X$ of dimension $m$, where $h$ is a smooth Hermitian metric such that $\omega=c_{1}(h)$ is the K\"ahler form of $X$.

Recall that a $quasi-plurisubharmonic$ (q.p.s.h.) function is an upper semi-continuous (usc) function which is locally the difference of a p.s.h. function and a smooth one. Any q.p.s.h. functions $\varphi$ on $X$ satisfies 
\begin{equation*}
dd^{c}\varphi+c\omega\geq 0
\end{equation*}
for some $c\geq 0$.
Denote by $PSH(X, \omega)$ the set of all q.p.s.h. functions $\varphi$ with $dd^{c}\varphi+\omega\geq 0$.
A subset $K\subset X$ is called $pluripolar$ if there exists a p.s.h. function $\varphi$ such that $K\subset\{\varphi=-\infty\}$.

Now we consider a non-pluripolar compact subset $K\subset X$
and a continuous function $q: K\rightarrow\mathbb{R}$.
Then the $extremal$  $function$ is defined to be the usc regularization of the function
\begin{equation}
V_{K,q}(x):=\sup\{\varphi(x)\in PSH(X,\omega):\varphi\leq q \quad \text{on} \quad K\}.
\end{equation}
This generalizes the notion of classical Siciak's extremal functions \cite{sj}.
Throughout this paper, we assume that $V_{K,q}$ is continuous.
The paper \cite{bb} gave a local regularity condition to make $V_{K,q}$ continuous. Then $V_{K,q}$ is equal to its usc regularization.
It follows from \cite{gz} that $V_{K,q}\in PSH(X, \omega)$.
So $\omega+dd^{c}V_{K,q}$ is a positive closed $(1,1)$-current which represents $c_{1}(L)$.
We write $V$ for $V_{K,q}$ for simplicity.

Recall that the Bergman kernel for $H^{0}(X, L)$ is the Schwartz kernel of the orthogonal projection from the space of global $L^{2}$-sections of $L$ onto $H^{0}(X, L)$. A natural inner product on $H^{0}(X, L)$ is defined by the following formula
\begin{equation*}
\langle s_{1}, s_{2}\rangle:=\int_{X}h_{n}(s_{1},s_{2})\omega^{m}, \quad s_{1}, s_{2}\in H^{0}(X, L^{n}).
\end{equation*}
We choose an orthonormal basis with respect to the inner product. Then we obtain the Bergman kernel for $ H^{0}(X, L^{n}) $. Note that the space of $ L^{2} $-sections of $L^{n}$ depends on the definition of inner products on $ H^{0}(X, L^{n}) $. Hence the Bergman kernel depends also on the definition.

Now we define a new inner product on $ H^{0}(X, L^{n}) $ which yields a weighted Bergman kernel. Let $ K, q $ be as before. Consider a positive measure $\mu$ with support on $ K $.
Note that $\mu(K)>0$.
\begin{definition}
The triple $(K,q,\mu)$ satisfies the Bernstein-Markov inequality of special type if we have 
\begin{equation*}
\max_{K}(h_{n}(s(x),s(x))e^{-2nq(x)})\leq M_{n}
\int_{K}h_{n}(s(x),s(x))e^{-2nq(x)}d\mu(x)
\end{equation*}
for all $s\in H^{0}(X, L^{n})$, where $M_{n}\leq n^{C_{0}}$
for some universal constant $C_{0}>0$.
\end{definition}
By the above Bernstein-Markov inequality, we have the following well-defined inner product on $ H^{0}(X, L^{n}) $
\begin{equation}
\langle s_{1}, s_{2}\rangle:=\int_{K}h_{n}(s_{1}(x),s_{2}(x))e^{-2nq(x)}d\mu(x).
\end{equation}
Note that if $\langle s,s\rangle=0$, then $h_{n}(s(x),s(x))=0$ on $K$. The identity theorem implies that $s$ is the zero section.
Now we choose an orthonormal basis $\{S_{nj}\}_{j=1}^{d_{n}}$
with respect to the above inner product. The weighted Bergman kernel function for 
$ H^{0}(X, L^{n}) $ is given by 
\begin{equation}
B_{n}(x):=\sum_{j=1}^{d_{n}}\|S_{nj}(x)\|_{h_{n}}^{2}.
\end{equation}
Here $d_{n}=\dim H^{0}(X, L^{n}), \|s(x)\|_{h_{n}}^{2}=h_{n}(s(x),s(x))$.
It is well-known \cite{mm} that 
\begin{equation*}
d_{n}=O(n^{m}).
\end{equation*}

At the end of this section, we recall the H\"ormander's $L^{2}$-estimate for $\bar\partial$ \cite{dj}.
\begin{theorem}
Let $(L,h)\rightarrow (X,\omega)$ be a singular Hermitian holomorphic line bundle over a complete K\"ahler manifold $X$ of dimension $m$. Let $K_{X}^{\star}$ be the dual of the canonical line bundle with the natural Hermitian metric $h_{0}$.
If there exists a continuous function $\lambda:X\rightarrow[0, \infty)$ such that $c_{1}(h)\geq \lambda\omega$, then for any form $g\in L_{m,1}^{2}(X, L, loc)$
satisfying 
\begin{equation*}
\bar\partial g=0, \quad \int_{X}\lambda^{-1}\|g\|^{2}\omega^{m}<\infty,
\end{equation*}
there exists $u\in L_{m,0}(X,L)$ with $\bar\partial u=g$ and
\begin{equation*}
\int_{X}\|u\|^{2}\omega^{m}\leq \int_{X}\lambda^{-1}\|g\|^{2}\omega^{m}.
\end{equation*}
If there exists $C>0$ such that 
\begin{equation*}
c_{1}(L,h)+c_{1}(K_{X}^{\star},h_{0})\geq C\omega,
\end{equation*}
then for any form $g\in L_{0,1}^{2}(X,L)$ with $\bar\partial g=0$, there exists $u\in L_{0,0}^{2}(X,L)$ satisfying
\begin{equation*}
\bar\partial u=g, \quad \int_{X}\|u\|^{2}\omega^{m}\leq \frac{1}{C}\int_{X}\|g\|^{2}\omega^{m}.
\end{equation*}
\end{theorem}

\section{Proof of the main theorem}
In this section the proof of the main theorem is divided into two steps. Set
\begin{equation*}
\Phi_{n}(x):=\sup \{\|s(x)\|_{h_{n}}: s\in H^{0}(X, L^{n}),\quad
\max_{K}(\|s(x)\|_{h_{n}}e^{-nq(x)})\leq 1\}.
\end{equation*}

Recall that the positive measure $\mu$ is in Definition 3.1.
First we prove the following 
\begin{proposition}
With the above notations and assumptions, we have the uniform estimate on $X$
\begin{equation*}
\frac{1}{\mu(K)d_{n}}\leq \frac{B_{n}(x)}{\Phi_{n}^{2}(x)}\leq n^{C_{0}}d_{n}. 
\end{equation*}
In particular, 
\begin{equation*}
|\frac{1}{2n}\log B_{n}(x)-\frac{1}{n}\log \Phi_{n}|=O(\frac{\log n}{n}).
\end{equation*}
\end{proposition}
\begin{proof}
Recall that $\{S_{nj}\}_{j=1}^{d_{n}}$ is the orthonormal basis of
$H^{0}(X, L^{n})$. We write $\{S_{j}\}$ for short.
The Bernstein-Markov inequality yields
\begin{equation*}
\max_{K}\|S_{j}\|_{h_{n}}\leq M_{n}^{\frac{1}{2}}.
\end{equation*}
Then
\begin{equation*}
\|S_{j}\|_{h_{n}}\leq M_{n}^{\frac{1}{2}}\Phi_{n}.
\end{equation*}
Hence 
\begin{equation*}
B_{n}=\sum_{j=1}^{d_{n}}\|S_{j}\|_{h_{n}}^{2}\leq M_{n}d_{n}\Phi_{n}^{2}\leq n^{C_{0}}d_{n}\Phi_{n}^{2}.
\end{equation*}

For the left inequality, we consider any section $s\in H^{0}(X, L^{n})$ satisfying $\max_{K}(\|s(x)\|_{h_{n}}e^{-nq(x)})\leq 1$.
Let $S_{n}(x,y)$ be the Bergman kernel.
Then 
\begin{equation*}
\begin{split}
\|s(x)\|_{h_{n}}^{2}&=\|\int_{K}h_{n}(y)(S_{n}(x,y), s(y))e^{-2nq(y)}d\mu(y)\|_{h_{n}(x)}^{2} \\
&=\| \sum_{j=1}^{d_{n}}\int_{K} h_{n}(y)(S_{j}(y), s(y))e^{-2nq(y)}d\mu(y) S_{j} (x)\|_{h_{n}(x)}^{2} \\
&\leq \| \sum_{j=1}^{d_{n}}\int_{K} \|S_{j}(y)\|_{h_{n}} \|s(y)\|_{h_{n}}e^{-2nq(y)}d\mu(y) S_{j} (x)\|_{h_{n}(x)}^{2} \\
&\leq \| \sum_{j=1}^{d_{n}}\int_{K} \|S_{j}(y)\|_{h_{n}} e^{-nq(y)}d\mu(y) S_{j} (x)\|_{h_{n}(x)}^{2} \\
&\leq \sum_{j=1}^{d_{n}}| \int_{K} \|S_{j}(y)\|_{h_{n}} e^{-nq(y)}d\mu(y) |^{2}\sum_{j=1}^{d_{n}}\|S_{j}\|_{h_{n}}^{2}\\
&\leq \mu(K)\sum_{j=1}^{d_{n}}\int_{K} \|S_{j}(y)\|_{h_{n}}^{2} e^{-2nq(y)}d\mu(y)B_{n}(x)\\
&=\mu(K)d_{n}B_{n}(x). \\
\end{split}
\end{equation*}
So $\Phi_{n}(x)^{2}\leq \mu(K)d_{n}B_{n}(x)$.
The proof is completed by using the fact that $d_{n}=O(n^{m})$.
\end{proof}

Now we prove the main theorem. Our proof is based on the original idea by Demailly \cite{dj1}.
\begin{proof}
By the above proposition, it suffices to show that
\begin{equation*}
|\frac{1}{n}\log \Phi_{n}-V|=O(\frac{\log n}{n})
\end{equation*}
uniformly on $X$.
Recall that $V:=V_{K,q}$ is defined in (1).

Let $\varphi:=\frac{1}{n}\log \|s\|_{h_{n}}$ for $s\in H^{0}(X, L^{n}) $.
Since 
\begin{equation*}
\max_{K}(\|s(x)\|_{h_{n}}e^{-nq(x)})\leq 1,
\end{equation*}
we deduce that
\begin{equation*}
\varphi\leq q \quad\text{on}\quad K.
\end{equation*}
Locally $s(x)=s_{\alpha}(x)e_{\alpha}^{n}$, then
\begin{equation*}
dd^{c}\varphi=\frac{1}{n}dd^{c}\log |s_{\alpha}|-dd^{c}\phi_{\alpha}.
\end{equation*}
So
\begin{equation*}
dd^{c}\varphi+\omega\geq 0.
\end{equation*}
By the definition of $V(x)$, $\varphi(x)\leq V(x)$.

Fix $x_{0}\in X$. Let 
\begin{equation*}
a=V(x_{0})-\frac{1}{n}<V(x_{0}).
\end{equation*}
By regularizations and translations, we can assume without loss of generality,
\begin{equation*}
\begin{split}
&\varphi\in PSH(X,\omega)\cap C^{\infty}(X), \\
&\sup_{K}(\varphi-q)<0,\\
&\varphi(x_{0})>V(x_{0})-\frac{1}{n}.
\end{split}
\end{equation*}
Denote by $B(x_{0},r)$ the open ball centered at $x_{0}$ with radius $r$. Since $(L,h)$ is Lipschitz,
we choose $r=\frac{1}{n^{l}}$ for some $l\geq 1$
such that
\begin{equation*}
\varphi(x)>V(x_{0})-\frac{1}{n} \quad\text{in}\quad B(x_{0},r)
\end{equation*}
and
\begin{equation*}
\sup_{x,y\in B(x_{0},r)}|h(x)-h(y)|\leq\tilde C\frac{1}{n^{l}}.
\end{equation*}
Here without confusions, we write locally $h$ for $\phi_{\alpha}$.

Let $\chi$ be a cut-off function with support in $B:=B(x_{0},\frac{1}{n^{l}})$ and $\chi\equiv 1$ on $B(x_{0},\frac{1}{2n^{l}})$ with $B\subset U_{\alpha}$.
Without loss of generality, we assume 
\begin{equation*}
B\cap U_{\beta}=\varnothing, \quad \forall\beta\neq\alpha.
\end{equation*}
Then $\chi e_{\alpha}^{n}$ is a global section of $L^{n}$ for all $n\geq 1$.

Recall that $K_{X}$ is the canonical line bundle of $X$.
There exists an integer $N_{1}$ such that $L^{N_{1}}\otimes K_{X}^{\star}$ is positive. Let $h_{1}=\{\phi_{1\alpha}\}$ be the smooth Hermitian metric of $L^{N_{1}}\otimes K_{X}^{\star}$.
Let $h_{2}=\{\phi_{2\alpha}\}$ be a singular Hermitian metric of $L^{N_{2}}$ for some $N_{2}>0$. Moreover, $h_{2}$ is smooth in $X\setminus{x_{0}}$ and its Lelong number $\nu(h_{2}, x_{0})\geq m$. That is 
\begin{equation*}
\liminf_{x\rightarrow x_{0}}\frac{h_{2}}{\log |x-x_{0}|}\geq m.
\end{equation*}
So near $x_{0}$ we have
\begin{equation*}
\|e_{\alpha}\|_{h_{2}}=e^{-h_{2}(x)}\gtrsim |x-x_{0}|^{-m}.
\end{equation*}
The choice of $h_{2}$ is possible when $N_{2}$ is large, since $L$ is positive. 

The section $\bar\partial\chi e_{\alpha}^{n}$ can be regarded as a $\bar\partial$-closed $(m-1)$-form with values in $L^{n}\otimes K_{X}^{\star}$. Note that the curvature form of $K_{X}^{\star}$ is 
\begin{equation*}
Ric(\omega)=-\partial\bar\partial\log \det\omega ,
\end{equation*}
which is smooth.
Let
\begin{equation*}
\psi_{n}:=(n-N_{1}-N_{2})(\varphi+\phi_{\alpha})+\phi_{1\alpha}+\phi_{2\alpha}.
\end{equation*}
Since
\begin{equation*}
dd^{c}\psi_{n}\geq dd^{c}\phi_{1\alpha}\geq\varepsilon_{0}\omega 
\end{equation*}
for some universal constant $\varepsilon_{0}>0$,
so we can apply H\"ormander's $L^{2}$-estimate theorem to the case with the metric $\psi_{n}$.
Hence there exists a smooth section $fe_{\alpha}^{n}$ such that
$\bar\partial f=\bar\partial\chi$ with the following estimate
\begin{equation*}
\begin{split}
&\int_{X}|f|^{2}e^{-2(n-N_{1}-N_{2})(\varphi+\phi_{\alpha})-2\phi_{1\alpha}-2\phi_{2\alpha}}dV_{\omega}\\
&\leq\frac{1}{\varepsilon_{0}}\int_{X}\|\bar\partial\chi\|^{2}e^{-2\psi_{n}}dV_{\omega}.
\end{split}
\end{equation*}
Since $\supp\bar\partial\chi\subset B\setminus B(x_{0},\frac{1}{2n^{l}})$,
$\psi_{n}$ is smooth in $B(x_{0},\frac{1}{2n^{l}})$.
We deduce that 
\begin{equation*}
\int_{X}\|\bar\partial\chi\|^{2}e^{-2\psi_{n}}dV_{\omega}<\infty.
\end{equation*}
Recall that the Lelong number of $\phi_{2\alpha}$ at $x_{0}$
satisfies $\nu(\phi_{2\alpha},x_{0})\geq m=\dim_{\mathbb{C}}X$, i.e. $e^{-2\phi_{2\alpha}(x)}\gtrsim |x-x_{0}|^{-2m}$ near $x_{0}$. But $|x-x_{0}|^{-2m}$ is not integrable near $x_{0}$, we must have $f(x_{0})=0$.

On the open ball $B$, $-\varphi<-V(x_{0})+\frac{1}{n}, -\phi_{\alpha}(x)\leq -\phi_{\alpha}(x_{0})+\frac{\tilde C}{n^{l}}$.
So
\begin{equation*}
\begin{split}
&-2n(\varphi+\phi_{\alpha})\\
&\leq 2n(-V(x_{0})+\frac{1}{n}-\phi_{\alpha}(x_{0})+\frac{\tilde C}{n^{l}})\\
&=-2nV(x_{0})+2-2n\phi_{\alpha}(x_{0})+\frac{2\tilde C}{n^{l-1}}.
\end{split}
\end{equation*}
There exists $C_{1}$ independent of $n$ such that
\begin{equation*}
\|\bar\partial\chi\|^{2}\leq C_{1}n^{2l} \quad\text{on}\quad B.
\end{equation*}
Hence
\begin{equation*}
\begin{split}
&\int_{X}\|\bar\partial\chi\|^{2}e^{-2\psi_{n}}dV_{\omega}\\
&\leq C_{2}n^{2l}e^{-2n(V_{x_{0}}+\phi_{\alpha}(x_{0}))}.\\
\end{split}
\end{equation*}
Set
\begin{equation*}
s_{\alpha}=se_{\alpha}^{n}=(\chi-f)e_{\alpha}^{n}\in H^{0}(X,L^{n}).
\end{equation*}
Note that $s(x_{0})=1$ and $s$ is a non-zero smooth section of $L^{n}$. Moreover, we have
\begin{equation*}
\begin{split}
&\int_{X}|s|^{2}e^{-2n(\varphi+\phi_{\alpha})}dV_{\omega}\\
&\leq C_{3}n^{2l}e^{-2n(V(x_{0})+\phi_{\alpha}(x_{0}))}.\\
\end{split}
\end{equation*}
Since $\varphi<q$ on $K$, we apply the mean-value inequality to the subharmonic function $s$, $\forall x\in K$,
\begin{equation*}
\begin{split}
&|s|^{2}e^{-2n\phi_{\alpha}}e^{-2nq}\\
&\leq C_{\delta}\int_{B(x,\delta)}|s(y)|^{2}e^{-2n(\varphi(y)+\phi_{\alpha}(y))}e^{2n(\phi_{\alpha}(y)-\phi_{\alpha}(x)+q(y)-q(x))}e^{2n(\varphi(y)-q(y))}dV_{\omega}(y).\\
\end{split}
\end{equation*}
Note that
\begin{equation*}
\begin{split}
&\lim_{\delta\rightarrow 0}\sup_{B(x,\delta)}|\phi_{\alpha}(y)-\phi_{\alpha}(x)+q(y)-q(x)|=0\\
&\lim_{\delta\rightarrow 0}\sup_{B(x,\delta)}(\varphi-q)<0.\\
\end{split}
\end{equation*}
Then we have
\begin{equation*}
\begin{split}
&|s|^{2}e^{-2n\phi_{\alpha}}e^{-2nq}\\
&\leq C_{4}n^{2l}e^{-2n(V(x_{0})+\phi_{\alpha}(x_{0}))}.\\
\end{split}
\end{equation*}
Set
\begin{equation*}
S=C_{4}^{-\frac{1}{2}}n^{-l}e^{n(V(x_{0})+\phi_{\alpha}(x_{0}))}se_{\alpha}^{n}.
\end{equation*}
Then
\begin{equation*}
\|S\|_{h_{n}}^{2}e^{-2nq}\leq C_{4}^{-1}n^{-2l}e^{2n(V(x_{0})+\phi_{\alpha}(x_{0}))}
|s|^{2}e^{-2n\phi_{\alpha}}e^{-2nq}\leq 1.
\end{equation*}
\begin{equation*}
\begin{split}
&\frac{1}{n}\log \|S\|_{h_{n}}(x_{0})\\
&=\frac{1}{n}\log(C_{4}^{-\frac{1}{2}}n^{-l}e^{n(V(x_{0})+\phi_{\alpha}(x_{0}))}|s(x_{0}|e^{-n\phi_{\alpha}(x_{0})})\\
&=-\frac{\log C_{4}}{2n}-l\frac{\log n}{n}+V(x_{0}).\\
\end{split}
\end{equation*}
Since $X$ is compact, the set of all $l$ is bounded.
Hence
\begin{equation*}
0\leq V(x)-\frac{1}{n}\log \Phi_{n}\leq C_{5}\frac{\log n}{n}
\end{equation*}
for all $x\in X$.
Then the proof is completed.
\end{proof}

\noindent
G. SHAO,
Institute of Mathematics, Academia Sinica, Taipei 10617, Taiwan.
{\tt guokuan@gate.sinica.edu.tw}

\end{document}